\documentclass[11pt,a4paper,reqno]{article}

\usepackage{graphics}   % graphic package for postscript figures
\usepackage{mathdots}
\usepackage{epsfig}
\usepackage{color}
\usepackage{anysize}
\usepackage{latexsym,amssymb,amsmath,amscd,amsthm,amsxtra}
\marginsize{2.5cm}{2.5cm}{1.0cm}{1.0cm}
\numberwithin{equation}{section}

\newtheorem{thm}{Theorem}[section]
\newcommand{\bth}{\begin{thm}}
\newtheorem{bdef}{Definition}
\newcommand{\brdef}{\begin{bdef}}
\newcommand{\erdef}{\end{bdef}}

\newtheorem{lem}{Lemma}[section]
\newcommand{\blem}{\begin{lem}}
\newcommand{\elem}{\end{lem}}

\newtheorem{alg}{Algorithm}[section]
\newcommand{\balg}{\begin{alg}}
\newcommand{\ealg}{\end{alg}}

\newtheorem{rem}{Remark}
\newcommand{\brem}{\begin{rem}}
\newcommand{\erem}{\end{rem}}

\newtheorem{exa}{Example}
\newcommand{\bexa}{\begin{exa}}
\newcommand{\eexa}{\end{exa}}

\newtheorem{corr}{Corollary}[section]
\newcommand{\bcor}{\begin{corr}}
\newcommand{\ecor}{\end{corr}}

\newcommand{\bproof}{\begin{proof}}
\newcommand{\eproof}{\end{proof}}

\def\({\left(}
\def\){\right)}
\def\vec{\mathop{\mathrm{vec}}}
\def\trace{\mathop{\mathrm{trace}}}
\def\Im{\mathop{\mathrm{Im}}}
\def\diag{\mathop{\mathrm{diag}}}
\def\max{\mathop{\mathrm{max}}}
\begin{document}

\title {\textrm{\textbf{On the conditioning of the matrix-matrix exponentiation}}}

\author{\normalsize{\textbf{Jo\~{a}o R. Cardoso$^{a}$\footnote{E-mail address of Jo\~ao R. Cardoso: jocar@isec.pt}, Amir Sadeghi $^{b}$\footnote{Corresponding author (E-mail address: drsadeghi.iau@gmail.com)} ,}} \\
\small{\textit{$^{a}$ Polytechnic Institute of Coimbra/ISEC, Coimbra -- Portugal, and}}\\
\small{\textit{Institute of Systems and Robotics,
 University of Coimbra, P\'{o}lo II, Coimbra -- Portugal}}\\
\small{\textit{$^{b}$Department of Mathematics, Robat Karim Branch, Islamic Azad University, Tehran, Iran.    }}
}

\date{}
\maketitle
\vspace {-.5cm}

\maketitle
\thispagestyle{empty}

\begin{abstract}

If ${A}$ has no eigenvalues on the closed negative real axis, and $B$ is arbitrary square complex, the matrix-matrix exponentiation is defined as $A^B:=e^{\log({A}){B}}$. This function arises, for instance, in Von Newmann's quantum-mechanical entropy, which in turn finds applications in other areas of science and engineering. Since in general $A$ and $B$ do not commute, this bivariate matrix function may not be a primary matrix function as commonly defined, which raises many challenging issues. In this paper, we revisit this function and derive new related results. Particular emphasis is given to its Fr\'echet derivative and conditioning. We present a general result on the Fr\'echet derivative of bivariate matrix functions with applications not only to the matrix-matrix exponentiation but also to other functions, such as the second order Fr\'echet derivatives and some iteration functions arising in matrix iterative methods. The numerical computation of the Fr\'echet derivative is discussed and an algorithm for computing the relative condition number of $A^B$ is proposed.  Some numerical experiments are included.

\end{abstract}

\textit{keywords}: Matrix-matrix exponentiation, Conditioning, Fr\'echet Derivative, Matrix exponential, Matrix logarithm

\section{Introduction}

Let $A$ be an $n\times n$ square complex matrix with no eigenvalues on the closed negative real axis $\mathbb{R}_0^-$ and let $B$ be an arbitrary square complex matrix of order $n$. The matrix-matrix exponentiation $A^B$ is defined as
\begin{equation}\label{AB}
A^B:=e^{\log({A}){B}},
\end{equation}
where $e^X$ stands for the exponential of the matrix $X$ and $\log(A)$ denotes the principal logarithm of $A$, i.e., the unique solution of the matrix equation $e^X=A$ whose eigenvalues lie on the open strip $\{ z\in \mathbb{C} :-\pi <\Im z<\pi\}$ of the complex plane; $\Im z$ stands for the imaginary part of $z$.

For background on matrix exponential, matrix logarithm and general matrix functions see \cite{Higham, Horn94} and the references therein. Note that although $A^B$ includes well-known matrix functions as particular cases (for instance, the matrix inverse and real powers of a matrix; see Lemma \ref{lema-basic} below), it is not, in general, a primary matrix function as defined in those books. Indeed, there may not exist a scalar single variable stem function associated to the matrix-matrix exponentiation. However, we can view $A^B$ as being an extension of the two variable function $x^y=e^{x\log y}$, but the lack of commutativity between $A$ and $B$ turns the extension of this function to matrices quite cumbersome. An interesting attempt to define the concept of bivariate matrix function as an operator is given in the monograph \cite{Kressner}. Although we refer to the matrix-matrix exponentiation as being a bivariate matrix function, it does not belong to the class of bivariate matrix functions defined in \cite{Kressner}. Here, $A^B$ can be regarded as a function from $\mathbb{C}^{n\times n}\times \mathbb{C}^{n\times n}$ to $\mathbb{C}^{n\times n}$ which assigns to each pair of matrices $(A,B)$ the $n\times n$ square complex matrix $A^B$.
Another way of defining the concept of matrix-matrix exponentiation would be
$$
^B\hspace*{-0.8mm}A:=e^{B\log({A})}.
$$
In Section \ref{basic}, some relationships between $A^B$ and $^B\hspace*{-0.8mm}A$ are pointed out (see, in particular, (iii) in Lemma \ref{lemma2}). However, our attention will be mainly focused on $A^B$. Analogue results follow straightforward for $^B\hspace*{-0.8mm}A$. A definition of the matrix-matrix exponentiation in a componentwise fashion is also possible, as used in \cite{Gentleman} to deal with some problems is Statistics. However, this latter definition is not considered in this work.

One of our goals is to investigate the sensitivity of the function $A^B$ to perturbations of first order in $A$ and $B$. A widely used tool to carry out this is the Fr\'echet derivative, which in turn allows the computation of the condition number of the function. In this work, we derive a general result on the Fr\'echet derivative of certain bivariate matrix functions (Theorem \ref{thm-LfAB}), which can be used to find easily an explicit formula for the Fr\'echet derivative of the matrix-matrix exponentiation in terms of the Fr\'echet derivatives of the matrix exponential and matrix logarithm. Formulae for the Fr\'echet derivatives of other bivariate matrix functions, such as iteration functions to the matrix square root (see \cite[Sec. 6.4]{Higham}) and to the matrix arithmetic-geometric mean (see \cite{Cardoso}), can also be obtained from the application of that result. The same holds for the second order Fr\'echet derivatives of primary matrix functions.

Given a map $f:\mathbb{C}^{n\times n}\times \mathbb{C}^{n\times n} \rightarrow\mathbb{C}^{n\times n}$, the Fr\'{e}chet derivative of $f$ at $(X,Y)$, with $X,Y\in\mathbb{C}^{n\times n}$, in the direction of $(E,F)$, where $E,F\in\mathbb{C}^{n\times n}$, is a linear operator $L_f(X,Y)$ that maps the ``direction matrix'' $(E,F)$ to $L_f(X,Y;E,F)$ such that
$$\lim_{(E,F)\rightarrow (0,0)}\frac{\|f(X+E,Y+F)-f(X,Y)-L_f(X,Y;E,F)\|}{\|(E,F)\|}=0.$$ The Fr\'{e}chet derivative of $f$ may not exist at $(X,Y)$, but if it does it is unique and coincides with the directional (or G\^ateaux) derivative of $f$ at $(X,Y)$ in the direction $(E,F)$. Hence, the existence of the Fr\'echet derivative guarantees that for any $E,F\in\mathbb{C}^{n\times n}$,
 $$L_f(X,Y;E,F)=\lim_{h\rightarrow 0}\frac{f(X+hE,Y+hF)-f(X,Y)}{h}.$$
 Any consistent matrix norm $\|.\|$ on $\mathbb{C}^{m\times n}$ induces the operator norm
 $$\ \|L_f(X,Y)\|:=\max_{(E,F)\neq 0}\,\frac{\|L_f(X,Y;E,F)\|}{\|(E,F)\|}.$$
 The (relative) condition number of $f$ at $(X,Y)$ is defined by
 \begin{equation}\label{cond-number}
\kappa_f(X,Y):=\frac{\|L_f(X,Y)\|\,\|(X,Y)\|}{\|f(X,Y)\|}.
\end{equation}
Hence, if an approximation to $L_f(X,Y;E,F)$ is known, then there exist numerical schemes to estimate $\|L_f(X,Y)\|$ (for instance, the power method on Fr\'{e}chet derivative proposed in \cite{Kenney}; see also \cite[Alg. 3.20]{Higham}) and then the condition number $\kappa_f(X,Y)$. As far as we know, we are the first to investigate the Fr\'echet derivative of the matrix-matrix exponentiation and its conditioning. In Section \ref{conditioning}, we discuss the efficient computation of $L_f(A,B;E,F)$, where $f(A,B):=A^B$, and propose a power method for estimating the Frobenius norm of $L_f(A,B)$ and then the corresponding condition number $\kappa_f(A,B)$. In the numerical experiments carried out in Section \ref{experiments} for several pairs of matrices $(A,B)$, two iterations of the power method suffices to estimate $\|L_f(A,B)\|_F$ (where $\|.\|_F$ stands for the Frobenius norm), with a relative error smaller than $10^{-3}$.

 Here, one uses the same notation to denote both the matrix norm and the induced operator norm. For more information on the Fr\'echet derivative and its properties see, for instance, \cite[Ch. X]{Bhatia} and \cite[Ch. 3]{Higham}. Note also that the pair $(E,F)$ corresponds, using matrix terminology, to the block matrix $\left[\begin{array}{c}
                             E \\
                             F
\end{array} \right]$. So the notation $\|(E,F)\|$ used above is clear.

To our knowledge, the terminology ``matrix-matrix exponentiation'' was firstly coined by Barradas and Cohen in \cite{Barradas}, where this function arises in a problem of Von Newmann's quantum-mechanical entropy. Some properties of the matrix-matrix exponentiation are addressed in \cite{Barradas}, for the particular case when $A$ is a normal matrix. We revisit some of those properties and derive new ones.

A particular case of the matrix-matrix exponentiation is the so called ``scalar-matrix exponentiation''. If $t$ is a complex number no belonging to $\mathbb{R}_0^-$, we can define $t^A$ as the function from $\mathbb{C}\times \mathbb{C}^{n\times n}$ to $\mathbb{C}^{n\times n}$ which assigns to each pair $(t,B)$ the $n\times n$ square complex matrix $t^A:=e^{\log t A}$.
This function appears in the definitions of matrix Gamma and Beta functions, which in turn can be applied to solving certain matrix differential equations \cite{Jodar1, Jodar2}. Gamma and Beta functions in matrix form are defined, respectively, as \cite{Jodar2}:
\begin{equation}\label{gamma}
\Gamma({A})=\int_{0}^{\infty}e^{-t}t^{{A}-{I}}dt,
\end{equation}
\begin{equation}\label{beta}
\mathcal{B}({A},{B})=\int_{0}^{1}t^{{A}-{I}}(1-t)^{{B}-{I}}dt.
\end{equation}
Our results apply easily to this particular case.\medskip

{\bf Notation:} $\|.\|$ denotes a subordinate matrix norm and $\|.\|_F$ the Frobenius norm; $\Im(z)$ is the imaginary part of the complex number $z$; $\sigma(A)$ is the spectrum of the matrix $A$; $A^\ast$ is the conjugate transpose of $A$, $L^\star_f(.)$ is the adjoint of the linear transformation $L_f(.)$. \medskip

The organization of the paper is as follows. In Section \ref{basic} we revisit some facts about the matrix-matrix exponentiation and add some related results not previously stated in the literature. A formula for the Fr\'echet derivative of certain bivariate matrix is proposed in Section \ref{frechet}. This formula is in turn used to derive a formula for the Fr\'echet derivative of the matrix-matrix exponentiation. It is also explained how it can be applied to the Fr\'echet derivative of well known bivariate functions. Section \ref{conditioning} is devoted to investigate the conditioning of the matrix-matrix exponentiation. In particular, an algorithm for estimating the relative condition number is propose. Its performance is illustrated by numerical experiments in Section \ref{experiments}. A few conclusions are drawn in Section \ref{conclusions}.

\section{Basic results}\label{basic}

In this section we present some theoretical results on the matrix-matrix exponentiation that can be derived from the properties of the much studied exponential and logarithm matrix functions.

According to the definition (\ref{AB}) and some well-known identities valid for the matrix exponential, we have
\begin{equation}\label{series}
{A}^{B}=\sum_{k=0}^{\infty}\frac{1}{k!}(\log( {A})\,{B})^{k},
\end{equation}
and
$$
{A}^{{B}}=\lim_{k\rightarrow \infty}\left( {I}+\frac{1}{k}\log({A}){B}\right)^{k}.
$$
In addition, ${A}^{{B}}$ can be considered as the solution of the matrix initial value problem
$$
\frac{d{X}(t)}{ dt}=(\log({A}){B}) {X}(t), \quad {X}(0)={I}.
$$

\begin{lem}\label{lema-basic}
If ${A}\in \mathbb{C}^{n\times n}$ has no eigenvalues on $\mathbb{R}_0^-$, ${B}$ is any square complex matrix, and $f(A,B)=A^B$, then the following properties hold:
\begin{enumerate}
    \item[(i)] ${A}^{0}={I}$ and ${{I}}^{{B}}={I}$; \\
    \vspace {-0.6cm}
     \item[(ii)] ${A}^{\alpha {I}}={A}^{\alpha}$, with $\alpha\in\mathbb{R}$. In particular, ${A}^{\frac{1}{2} {I}}={A}^{\frac{1}{2}}$ and ${A}^{- {I}}={A}^{-1};$ \\
    \vspace {-0.6cm}
    \item[(iii)] If the eigenvalues of $\log(A)B$ satisfy $-\pi<\Im(\lambda)<\pi$, then ${A}^{({BC})}=({A}^{{B}})^{{C}}$;\\
    \vspace {-0.6cm}
    \item[(iv)] ${A}^{-{B}}{A}^{{B}}={A}^{{B}}{A}^{-{B}}={I}$, therefore $({A}^{B})^{-1}={A}^{-{B}}$;\\
    \vspace {-0.6cm}
    \item[(v)] $({A}^{{B}})^{\ast}=\, ^{{B}^{\ast}}\hspace*{-0.2cm}{A}^{\ast}$, where $X^\ast$ stands for the conjugate transpose of $X$; \\
        \vspace {-0.6cm}
     \item[(vi)] If $S$ is an invertible matrix then $f(SAS^{-1},SBS^{-1})=S\, f(A,B)\, S^{-1}$.
\end{enumerate}
\end{lem}
\begin{proof}
Immediate consequence from properties of matrix exponential and matrix logarithm. See \cite{Higham,Horn94}.
\end{proof}

The following example shows that explicit formulae for matrix-matrix exponentiation may involve complicated expressions, even for the case $2\times 2$, with $A$ normal.

\begin{exa}
{\rm Let ${A}=\bigl[\begin{smallmatrix}
               a & b \\
               -b & a \\
\end{smallmatrix} \bigr]$ be a nonsingular normal matrix and ${B}=\bigl[\begin{smallmatrix}
                             \alpha & \beta \\
                             \gamma & \delta \\
\end{smallmatrix} \bigr]$ be an arbitrary matrix. Our aim is to find a closed expression for $A^B$. The eigenvalues and eigenvectors of ${A}$ are displayed in matrices $D$ and $V$, respectively:
$${D}={\diag}(\lambda_{_{1}},\lambda_{2})=\left[\begin{array}{cc}
      a-ib & 0 \\
      0 & a+ib \\
    \end{array}\right], \quad {V}=\left[\begin{array}{rr}
                             i & -i \\
                             1 & 1 \\
                           \end{array}\right],
$$
where $a-ib=\overline{z}=re^{-i\theta}$ and $a+ib=z=re^{i\theta}$  for $-\pi<\theta\leq\pi$. It is clear that in the sense of polar notation, we have $r^{2}=a^{2}+b^{2}$ and $\theta=\arctan(b/a)$ (provided that $a\neq 0$). Therefore, the logarithm of ${A}$ can be evaluated by the decomposition $\log({A})={V}\log({D}){V}^{-1}$ as following:
$$
\log({A})=\left[\begin{array}{rr}
                             i & -i \\
                             1 & 1 \\
                           \end{array}\right]
                           \left[\begin{array}{cc}
      \log(\overline{z}) & 0 \\
      0 & \log(z) \\
    \end{array}\right]\left[\begin{array}{rr}
                             i & -i \\
                             1 & 1 \\
                           \end{array}\right]^{-1}=\left[\begin{array}{cc}
                             \log(r) & -\theta \\
                             \theta & \log(r) \\
                          \end{array}\right].
$$
Hence, multiplying the matrices $\log({A})$ and ${B}$,
$$
\log({A}){B}=\left[\begin{array}{cc}
                             \alpha\log(r)-\theta\gamma & \beta\log(r)-\theta\delta \\
                             \gamma\log(r)+\theta\alpha & \delta\log(r)+\theta\beta \\
                           \end{array}\right].
$$
It is known that the exponential of an $2\times2$ matrix ${M}=\bigl[\begin{smallmatrix}
               m_{11} & m_{12} \\
               m_{21} & m_{22} \\
\end{smallmatrix} \bigr]$, can be explicitly obtained by the following relation (see \cite{Rowland}):
$$
e^{{M}}=\frac{1}{\Omega}\left[\begin{array}{cc}
          e^{\frac{m_{11}+m_{22}}{2}}\left[\Omega\cosh(\frac{\Omega}{2})+(m_{11}-m_{22})\sinh(\frac{\Omega}{2})\right]   & 2m_{12}e^{\frac{m_{11}+m_{22}}{2}}\sinh(\frac{\Omega}{2}) \\
          2m_{21}e^{\frac{m_{11}+m_{22}}{2}}\sinh(\frac{\Omega}{2})   &  e^{\frac{m_{11}+m_{22}}{2}}\left[\Omega\cosh(\frac{\Omega}{2})+(m_{22}-m_{11})\sinh(\frac{\Omega}{2})\right] \\
                           \end{array}\right].
$$
where, $\Omega=\sqrt{(m_{11}-m_{22})^{2}+4m_{12}m_{21}}$. Consequently, an explicit formula for ${A}^{{B}}$ can be obtained via substituting:
$$m_{11}=\log(r^{\alpha})-\theta\gamma, \quad m_{12}=\log(r^{\beta})-\theta\delta, $$
$$m_{21}=\log(r^{\gamma})+\theta\alpha, \quad m_{22}= \log(r^{\delta})+\theta\beta. $$
}
\end{exa}

\medskip
As mentioned before, some facts about the matrix-matrix exponentiation have been reported in \cite{Barradas}, under the assumption of $A$ being normal. One of them is  revisited in (i) of the next lemma. However, the relationships between $A^B$ and $^B\hspace*{-0.8mm}A$, and their spectra, stated in (ii) and (iii) of the following lemma are new.

\begin{lem}\label{lemma2}
If ${A}\in \mathbb{C}^{n\times n}$ has no eigenvalues on $\mathbb{R}_0^-$, and ${B}$ is any square complex matrix, then the following properties hold:
\begin{enumerate}
    \item[(i)] $A^B$ and $^B\hspace*{-0.8mm}A$ have the same spectra;\\
    \vspace {-0.6cm}
    \item[(ii)] If $A$ and $B$ commute, and have spectra $\sigma(A)=\{\alpha_1,\ldots,\alpha_n\}$, $\sigma(B)=\{\beta_1,\ldots,\beta_n\}$, then the spectrum of $A^B$ (or $^B\hspace*{-0.7mm}A$) is given by $\{\alpha_{i_1}^{\beta_{j_1}},\ldots,\alpha_{i_n}^{\beta_{j_n}}\}$ for some permutations $\{i_1,\ldots,i_n\}$ and $\{j_1,\ldots,j_n\}$ of the set $\{1,\ldots,n\}$; \\
    \vspace {-0.6cm}
    \item[(iii)] $B\,A^B=\,^B\hspace*{-0.7mm}A\,B.$
\end{enumerate}
\end{lem}

\begin{proof}

\begin{enumerate}
\item[(i)] See \cite[Thm. 3.2]{Barradas}. This follows immediately from the classical result of matrix theory stating that when $X$ and $Y$ are square matrices, both products $XY$ and $YX$ have the same spectra (see, for instance, Theorem 1.3.20 and Problem 9 in \cite{Horn85}).
\item[(ii)] Since $A$ and $B$ commute, $\log(A)$ and $B$ also commute. Hence, the results follows from the fact that
$$\sigma\left(\log(A)B\right)\subset\left\{\log(\alpha_i)\beta_j:\ i,j=1,\ldots,n\,\right\}.$$
\item[(iii)] Immediate consequence of the identity $Ye^{XY}=e^{YX}Y$, that is valid for any square complex matrices of order $n$; see \cite[Cor. 1.34]{Higham}.
\end{enumerate}
\end{proof}

One important implication of the statement (iii) of Lemma \ref{lemma2} is that when $B$ in nonsingular, $^B\hspace*{-0.8mm}A$ can be computed easily from $A^B$:
$$^B\hspace*{-0.8mm}A=B\, A^B\,B^{-1}.$$

A natural way of computing the matrix-matrix exponentiation $A^B$ is to first evaluate $\log(A)$ and then the exponential of $\log(A)B$. Matrix exponential and logarithm are much studied functions and one can found many methods for computing them in the literature. The most popular method to the matrix exponential is the so-called scaling and squaring method combined with Pad\'e approximation, that has been investigated and improved by many authors; see for instance
\cite[Ch. 10]{Higham} and the references therein and also the more recent paper \cite{Mohy09} that includes the algorithm where the \texttt{expm} function of recent versions of MATLAB is based on. The MATLAB function \texttt{logm} implements the algorithm provided in \cite{Mohy12,Mohy13}, which is an improved version of the inverse scaling and squaring with Pad\'e approximants method proposed in \cite{Kenney}. Other methods for approximating these functions include, for instance, the Taylor polynomial based methods for the matrix exponential proposed in \cite{Sastre} and the iterative transformation-free method of \cite{Cardoso} for the matrix logarithm.

A topic that needs further research is the development of algorithms for the matrix-matrix exponentiation that are less expensive than the computation of one matrix exponential and one matrix logarithm plus a matrix product. This seems to be a very challenging issue, especially when $B$ does not commute with $A$. Of course, for some particular cases of the matrix-matrix exponentiation (e.g., the matrix square root, matrix $p$-th roots, the matrix inverse) there are more efficient methods that do not involve the computation of matrix exponentials and logarithms. This problem becomes easier even in the more general case when $A$ and $B$ commute. This is because both matrices may share the same Schur decomposition which reduces considerably the computational effort.

\section{The Fr\'echet derivative of bivariate matrix functions}\label{frechet}

 A key result, very useful from  both theoretical and computational perspectives, related with the Fr\'echet derivative of a primary matrix function $\phi:\mathbb{C}^{n\times n} \rightarrow\mathbb{C}^{n\times n}$, states that
\begin{equation}\label{block-1}
\phi\left(\left[
		\begin{array}{cc}
		 X & E  \\
		 0 & X
		\end{array}
		\right]\right)=\left[
		\begin{array}{cc}
		 \phi(X) & L_\phi(X,E)  \\
		 0 & \phi(X)
		\end{array}
		\right],
\end{equation}
where $\phi$ is a scalar complex function $2n-1$ times continuously differentiable on an open subset containing the spectrum of $X$, the matrix $E\in\mathbb{C}^{n\times n}$ is arbitrary and $L_\phi(X,E)$ denotes the Fr\'echet derivative of $\phi$ at $X$ in the direction of $E$ (see \cite[Thm. 2.1]{Mathias} and \cite[Eq. (3.16)]{Higham}).

Next theorem extends the identity (\ref{block-1}) to certain bivariate matrix functions.

\begin{thm}\label{thm-LfAB}
Let $X=\left[x_{ij}\right]_{i,j},\,Y=\left[y_{ij}\right]_{i,j},\,E,F\in \mathbb{C}^{n\times n}$ and assume that $f(X,Y)\in\mathbb{C}^{n\times n}$
is a bivariate matrix function with partial derivatives $\frac{\partial f}{\partial x_{ij}}$ and $\frac{\partial f}{\partial y_{ij}}$ being continuous functions on an open subset ${\mathcal S}\subset \mathbb{C}^{n\times n}\times \mathbb{C}^{n\times n}$ containing $(X,Y)$. If the curves $X(t):=X+tE$ and $Y(t):=Y+tF$ are differentiable at $t=0$, with $\left(X(t),Y(t)\right)\in {\mathcal S}$ for all $t$ in a certain neighborhood of $0$, $f$ maps $2\times 2$--block upper triangular matrices to $2\times 2$--block upper triangular, then
 \begin{equation}\label{block-2}
f\left(\left[
		\begin{array}{cc}
		 X & E  \\
		 0 & X
		\end{array}
		\right],\,
\left[\begin{array}{cc}
		 Y & F  \\
		 0 & Y
		\end{array}
		\right]\right)=
\left[
		\begin{array}{cc}
		 f(X,Y) & L_f(X,Y;E,F)  \\
		 0 & f(X,Y)
		\end{array}
		\right].
\end{equation}
\end{thm}

\begin{proof}
 Since the $2n^2$ partial derivatives $\frac{\partial f}{\partial x_{ij}}$ and $\frac{\partial f}{\partial y_{ij}}$ exist and are continuous on the open subset ${\mathcal S}$, the Fr\'echet derivative of $f$ on ${\mathcal S}$ exists (see \cite[Sec. 3.1]{Cheney}) and thus coincides with the G\^{a}teaux derivative.

Assuming that $X(t)$ and $Y(t)$ are differentiable at $t=0$ and $\left(X(t),Y(t)\right)\in {\mathcal S}$ for all $t$ in a certain neighborhood of $0$, we shall prove below that an analogue identity to the one in \cite[Eq. (1.1)]{Mathias} (see also \cite[Thm. 3.6]{Higham}) holds for our function $f$, that is,
\begin{equation}\label{eq-mathias}
f\left(\left[
		\begin{array}{cc}
		 X & X'(0)  \\
		 0 & X
		\end{array}
		\right],\,
\left[\begin{array}{cc}
		 Y & Y'(0)  \\
		 0 & Y
		\end{array}
		\right]\right)=
\left[
		\begin{array}{cc}
		 f(X,Y) & \left.\frac{d}{dt}\right|_{t=0} f\left(X(t),Y(t)\right)  \\
		 0 & f(X,Y)
		\end{array}
		\right].
\end{equation}
Indeed, denoting
$$U=\left[\begin{array}{cc}
		 I & I/\epsilon  \\
		 0 & I
		\end{array}
		\right],$$
with $\epsilon\neq 0$, we have
\begin{eqnarray*}
&&f\left(\left[
		\begin{array}{cc}
		 X(0) & \frac{X(\epsilon)-X(0)}{\epsilon}  \\
		 0 & X(0)
		\end{array}
		\right],\,
\left[\begin{array}{cc}
		 Y(0) & \frac{Y(\epsilon)-Y(0)}{\epsilon}  \\
		 0 & Y(0)
		\end{array}
		\right]\right)
=\\
&&\quad\qquad=U\,f\left(U^{-1}\left[
		\begin{array}{cc}
		 X(0) & \frac{X(\epsilon)-X(0)}{\epsilon}  \\
		 0 & X(0)
		\end{array}
		\right]\,U,\,
U^{-1}\left[\begin{array}{cc}
		 Y(0) & \frac{Y(\epsilon)-Y(0)}{\epsilon}  \\
		 0 & Y(0)
		\end{array}
		\right]\,U\right)U^{-1}\\
&&\quad\qquad=U\,f\left(\left[
		\begin{array}{cc}
		 X(0) & 0  \\
		 0 & X(\epsilon)
		\end{array}
		\right],\,
\left[\begin{array}{cc}
		 Y(0) & 0  \\
		 0 & Y(\epsilon)
		\end{array}
		\right]\right)U^{-1}\\
&&\quad\qquad=
U\, \left[\begin{array}{cc}
		 e^{\log(X(0))\,Y(0)} & 0  \\
		 0 & e^{\log(X(\epsilon))\,Y(\epsilon)}
		\end{array}
		\right]\,U^{-1}\\
&&\quad\qquad=\left[\begin{array}{cc}
		 f(X,Y) & \frac{f\left(X(\epsilon),Y(\epsilon)\right)-f(X,Y)}{\epsilon}  \\
		 0 & f\left(X(\epsilon),Y(\epsilon)\right)
		\end{array}
		\right],
\end{eqnarray*}
from which the result follows by evaluating the limit of the above matrices when $\epsilon \rightarrow 0$.
\end{proof}

An explicit formula for the Fr\'echet derivative of the matrix-matrix exponentiation, in terms of the Fr\'echet derivatives of the matrix exponential and matrix logarithm, is given in the next corollary.

\begin{corr}\label{cor1}
Let ${\mathcal S}$ be the open subset formed by all pairs $(A,B)$ with $A$ having no eigenvalues on $\mathbb{R}_0^-$ and $B$ arbitrary. Denoting  $f(A,B):=A^B$, it holds
\begin{equation}\label{explicit}
L_f(A,B;E,F)=L_{\exp}\big( \log(A)\,B;\log(A)F+L_{\log}(A;E)\,B\big),
\end{equation}
where $L_{\exp}$ and $L_{\log}$ stand for the Fr\'echet derivatives of the matrix exponential and matrix logarithm, respectively.
\end{corr}

\begin{proof}
It easy to check that the conditions of Theorem \ref{thm-LfAB} are met. Then the result follows immediately from the identities
\begin{eqnarray*}
f\left(\left[
		\begin{array}{cc}
		 A & E  \\
		 0 & A
		\end{array}
		\right],\,
\left[\begin{array}{cc}
		 B & F  \\
		 0 & B
		\end{array}
		\right]\right)
&=&
e^{\log\left(\left[
		\begin{array}{cc}
		 A & E  \\
		 0 & A
		\end{array}
		\right]\right)\,
\left[\begin{array}{cc}
		 B & F  \\
		 0 & B
		\end{array}
		\right]}\\
&=&
e^{\left[
		\begin{array}{cc}
		 \log(A) & L_{\log}(A;E)  \\
		 0 & \log(A)
		\end{array}
		\right]\,
\left[\begin{array}{cc}
		 B & F  \\
		 0 & B
		\end{array}
		\right]}\\
&=&
\left[\begin{array}{cc}
		 e^{\log(A)\,B} & L_{\exp}\left( \log(A)\,B;\log(A)F+L_{\log}(A;E)\,B\right)  \\
		 0 & e^{\log(A)\,B}
		\end{array}
		\right].
\end{eqnarray*}
\end{proof}

If $A=tI$, with $t$ not belonging to the closed negative real axis, then $(tI)^B=t^B$ is the scalar-matrix exponentiation which arises in matrix Beta and Gamma functions defined in (\ref{beta}) and (\ref{gamma}), respectively. The Fr\'echet derivative in this special case can be written as
$$L_f(t,B;\epsilon,F)=L_{\exp}\left(\log(t)A;\log(t)F+B\epsilon/t\right).$$
For $B=\alpha\,I$, with $\alpha\in\mathbb{R}$, the matrix-matrix exponentiation $f$ reduces to the single variable matrix function $f(A)=A^\alpha$ of real powers of $A$, which has been addressed recently in \cite{Higham11,Higham13}. Provided that $\alpha$ is not affected by any kind of perturbation (that is, $F=0$), (\ref{explicit}) reduces to formula (2.4) in \cite{Higham11}, that has been obtained using other techniques. Note that while  (\ref{explicit}) covers the case when $\alpha$ is perturbed, formula (2.4) in \cite{Higham11} does not.

In addition to the result of the previous corollary, the identity (\ref{block-2}) provides alternative means for obtaining closed expressions for the Fr\'echet derivatives of other known bivariate matrix functions. For instance, many iteration functions for approximating the matrix square root (\cite{Higham97}, \cite[Ch. 6]{Higham}) or, more generally, for the matrix $p$-th root \cite{Guo, Iannazzo} are of the form
$$g(X,Y)=\left[\begin{array}{c}
		 g_1(X,Y)  \\
		 g_2(X,Y)
		\end{array}
		\right],$$
where $g_1$ and $g_2$ satisfy some smooth requirements. Since
$$L_g(X,Y;E,F)=\left[\begin{array}{c}
		 L_{g_1}(X,Y;E,F)  \\
		 L_{g_2}(X,Y;E,F)
		\end{array}
		\right],$$
a closed expression for $L_{g_i}(X,Y;E,F)$ ($i=1,2$) follows from (\ref{block-2}). The same relationship applies to the matrix arithmetic-geometric mean iteration \cite{Cardoso, Stickel} and to find expressions for the second order Fr\'echet derivatives \cite[Ch. X]{Bhatia} of primary matrix functions.

Closed formulae for the Fr\'echet derivatives of matrix exponential and matrix logarithm are available in the literature. One of the most known for the matrix exponential is the integral formula
\begin{equation}\label{frechet-exp}
L_{\exp}(A,E)=\int_0^1\, e^{A(1-t)}Ee^{At}\ dt
\end{equation}
(see \cite{VanLoan77} and \cite[Ch. 10]{Higham}). Another formula, involving the vectorization of the Fr\'echet derivative, is
\begin{equation}\label{frechet-vec-exp}
\vec\left(L_{\exp}(A,E)\right)=K_{\exp}(A)\,\vec(E),
\end{equation}
where $\vec(.)$ stands for the operator that stacks the columns of $E$ into a long vector of size $n^2\times 1$, and
$$K_{\exp}(A)=\left(I\otimes e^A\right)\,\psi\left(A\oplus(-A)\right)\in \mathbb{C}^{n^2\times n^2},$$
with $\psi(x)=(e^x-1)/x$. The symbols $\otimes$ and $\oplus$ denote the Kronecker product and the Kronecker sum, respectively. Other representations for $K_{\exp}(A)$ are available in \cite[Eq. (10.3)]{Higham}; see also \cite{Najfeld,Kenney}.

An integral representation of the Fr\'echet derivative of the matrix logarithm is
\begin{equation}\label{frechet-log}
L_{\log}(A,E)=\int_0^1\, \left(t(A-I)+I\right)^{-1}\,E\, \left(t(A-I)+I\right)^{-1}\ dt
\end{equation}
(see \cite{Dieci} and \cite[Ch. 11]{Higham}). Vectorizing (\ref{frechet-log}) yields
\begin{equation}\label{frechet-vec-log}
\vec\left(L_{\log}(A,E)\right)=K_{\log}(A)\,\vec(E),
\end{equation}
where
$$K_{\log}(A)=\int_0^1 \left(t(A-I)+I\right)^{-T}\otimes \left(t(A-I)+I\right)^{-1}\ dt \in \mathbb{C}^{n^2\times n^2}.$$
Gathering the formulae above, a vectorization of the Fr\'echet derivative of the matrix-matrix exponentiation $f(A,B)=A^B$ can be given by
\begin{equation}\label{frechet-vec-f}
\vec\left(L_{f}(A,B;E,F)\right)=K_{f}(A,B)\,\left[\begin{array}{c}
		 \vec(E) \\ \vec(F)
		\end{array}
		\right],
\end{equation}
where
\begin{equation}\label{kf}
K_f(A,B)=K_{\exp}\left(\log(A)\,B\right)\,\big[\left(B^T\otimes I\right)\,K_{\log}(A)\quad I\otimes \log(A)\big].
\end{equation}

Fr\'echet derivatives allow us to understand how the function $f(A,B)=A^B$ behaves when both $A$ and $B$ are subject to small perturbations. Suppose now that $A$ does not suffer any kind of perturbation but $B$ does. Now just $B$ is regarded as a variable and similar perturbed results to those of matrix exponential are valid, as shown below in Theorem \ref{perturb1}.
%The proofs are omitted because they follow easily from (\ref{explicit}) with $E=0$.

\begin{thm}\label{perturb1}
Assume that $A$ has no eigenvalue on $\mathbb{R}_0^-$. For any ${B}_{1}, {B}_{2}\in \mathbb{C}^{n\times n}$, the following relation holds:
\begin{equation}
\|{A}^{{B}_{1}}-{A}^{{B}_{2}}\|\leq\|{B}_{1}-{B}_{2}\|e^{\max\{\|\log({A}){B}_{1}\|,\|\log({A}){B}_{2}\|\}}.
\end{equation}
\end{thm}

\begin{proof}
From the theory of the matrix exponential, it is straightforward that
\begin{equation}\label{2-16}
{A}^{({B+E})t}={A}^{{B}t}+\int_{0}^{t}{A}^{{B}(t-s)}{E}{A}^{({B+E})s}ds
\end{equation}
(see \cite{Bellman}).
Let us consider ${B}={B}_{1}$, ${B}_{2}={B_{1}+E}$ and $t=1$ in (\ref{2-16}). Hence, we have
$${A}^{{B}_{2}}={A}^{{B}_{1}}+\int_{0}^{1}{A}^{{B}_{1}(1-s)}({B}_{1}-{B}_{2}){A}^{{B}_{2}s}ds.$$
Therefore, taking norms, one has
$$\begin{array}{rcl}
\|{A}^{{B}_{1}}-{A}^{{B}_{2}}\|&\leq& \left\|\int_{0}^{1}{A}^{{B}_{1}(1-s)}({B}_{1}-{B}_{2}){A}^{{B}_{2}s}ds\right\|\\\\
&\leq & \|{B}_{1}-{B}_{2}\|\int_{0}^{1}\left\|{A}^{{B}_{1}(1-s)}{A}^{{B}_{2}s}\right\|ds\\\\
&=& \|{B}_{1}-{B}_{2}\|\int_{0}^{1}\left\|e^{\log({A}){B}_{1}(1-s)}e^{\log({A}){B}_{2}s}\right\|ds\\\\
&\leq&\|{B}_{1}-{B}_{2}\|\int_{0}^{1}e^{\|\log({A}){B}_{1} \|(1-s)}e^{\|\log({A}){B}_{2}\|s}ds\\\\
&\leq&\|{B}_{1}-{B}_{2}\|\,e^{\max\{\|\log({A}){B}_{1}\|,\|\log({A}){B}_{2}\|\}}.
\end{array}
$$
\end{proof}

\section{Conditioning of the matrix-matrix exponentiation}\label{conditioning}

From now on, we will consider the Frobenius norm only. However, with appropriate modifications, some results can be adapted to other norms. The key factor for evaluating the condition number $k_f(A,B)$ is the norm of the operator $L_f(A,B)$. In this section, we first present an upper bound to such a norm and then a power method for its estimation. The notation $f(A,B)=A^B$ is used again.

\begin{thm}\label{thm-bound}
Assume that the conditions of Corollary \ref{cor1} are valid. With respect to the Frobenius norm, the following inequality holds:
\begin{equation}\label{bound}
\left\|L_f(A,B)\right\|_F\leq e^{\|\log(A)\|_F\|B\|_F}\sqrt{\left\|L_{\log}(A)\right\|^2_F\|B\|_F^2+\|\log(A)\|_F^2}.
\end{equation}
\end{thm}

\begin{proof}
For $M:=\log(A)F+L_{\log}(A;E)\,B$, we have
\begin{eqnarray*}
\|M\|_F & \leq & \|\log(A)\|_F \|F\|_F+ \left\|L_{\log}(A)\right\|_F\|B\|_F\|E\|_F  \\
        & \leq & \big[\left\|L_{\log}(A)\right\|_F\|B\|_F\quad \|\log(A)\|_F\big]
    \left\|\left[\begin{array}{c}
		 E \\ F
		\end{array}
		\right]\right\|_F.
\end{eqnarray*}

By (\ref{explicit}),
\begin{eqnarray*}
\left\|L_f(A,B;E,F)\right\|_F & = &  \left\|L_{\exp}\left( \log(A)\,B;M\right)\right\|_F\\
            & = &  \left\|\int_0^1 e^{\log(A)\,B (1-t)}M  e^{\log(A)\,B t}\ dt\right\|_F\\
            & \leq & \|M\|_F \int_0^1  e^{\|\log(A)\,B\|_F}\ dt\\
            & \leq & \|M\|_F e^{\|\log(A)\,B\|_F}.
\end{eqnarray*}
Hence
\begin{eqnarray*}
\left\|L_f(A,B)\right\|_F &=& \max_{\|(E,F)\|_F=1}\left\|L_f(A,B;E,F)\right\|_F \\
                & \leq & e^{\|\log(A)\,B\|_F} \left\|\left[\left\|L_{\log}(A)\right\|_F\|B\|_F\quad \|\log(A)\|_F\right]\right\|_F \\
                & \leq & e^{\|\log(A)\|_F\|B\|_F}\sqrt{\left\|L_{\log}(A)\right\|^2_F\|B\|_F^2+\|\log(A)\|_F^2}.
\end{eqnarray*}
\end{proof}

If $\|A-I\|_F<1$, one can find an upper bound for the factor $\|\log(A)\|_F$ in the right hand side of (\ref{bound}) as follows:
$$\|\log({A})\|_F\leq \sum_{k=1}^{\infty}\frac{\|{A}-{I}\|_F^{k}}{k}\leq \|{A}-{I}\|_F\sum_{k=0}^{\infty}\|{A}-{I}\|_F^{k}\leq \frac{\|{A}-{I}\|_F}{1-\|{A}-{I}\|_F}. $$
In the general case, it is hard to bound $\|\log(A)\|_F$, which can be infinitely large. However, since the logarithm function increases in a very slow fashion, in practice the values attained by $\|\log(A)\|_F$ can be considered small. For instance, its largest value for the ten  matrices considered in the numerical experiments in Section \ref{experiments} is about $50$ (see bottom-left plot in Figure \ref{figura}).

For better estimates to $\left\|L_f(A,B)\right\|_F$, we propose below a particular power method for the matrix-matrix exponentiation using the framework of \cite[Alg. 3.20]{Higham}. Before stating the detailed steps of the methods, we shall address two important issues raised by its implementation. The first one is the computation of the Fr\'echet derivative $L_f(A,B;E,F)$ and the second one is how to find the adjoint operator $L^\star$ with respect to the Euclidean inner product $\langle X,Y\rangle=\trace(Y^\ast X)$. We recall that the matrix of our linear operator $L_f(A,B)$ is not square which means that its expression is not so simple to obtain as in the square case, where one just needs to take the conjugate transpose of the argument (check the top of p. 66 in \cite{Higham}).

 About the first issue, and attending to the developments carried out in Section \ref{frechet}, we will use formula (\ref{explicit}). For the computation of $L_{\exp}$ we consider \cite[Alg. 6.4]{Mohy09}, and for the computation of $\log$ and $L_{\log}$ we use \cite[Alg. 5.1]{Mohy13} (without the computation of $L_{\log}^{\star}$). We can, alternatively, use
  \begin{equation}\label{frechet1}
  L_f(A,B;E,F)=\left(e^{\log\left(\left[
		\begin{array}{cc}
		 A & E  \\
		 0 & A
		\end{array}
		\right]\right)\,
\left[\begin{array}{cc}
		 B & F  \\
		 0 & B
		\end{array}
		\right]}\right)_{1,2},
\end{equation}
(this should be read as: the Fr\'echet derivative is the block $(1,2)$ of the resulting matrix in the right-hand side; see (\ref{eq-mathias})), but this formula is more expensive than (\ref{explicit}), even if we exploit the particular structure of the two block matrices in the right-hand side of (\ref{frechet1}). More disadvantages of formulae like (\ref{frechet1}) are mentioned in \cite{Mohy12,Mohy13}.

Now we focus on finding a closed expression for the adjoint operator $L^\star_f(A,B)$, where $f(A,B)=A^B$. According to the theory of adjoint operators (see, for instance, \cite{Friedberg}), one needs to look for the unique operator
\begin{center}
$\begin{array}{rccl}
L_f^\star(A,B) :& \mathbb{C}^{n\times n}  & \longrightarrow & \mathbb{C}^{n\times n}\times \mathbb{C}^{n\times n}\\
            & W & \longmapsto &   L_f^\star(A,B;W),\\
\end{array}$
\end{center}
such that
$$\vec\left(L_f^\star(A,B;W)\right)=K^\ast_f(A,B)\,\vec(W),$$
where $K^\ast_f(A,B)$ is the conjugate transpose of (\ref{kf}). Since
$$\vec\big(L_{\exp}\left((\log(A)B)^{\ast};W\right)\big)=K_{\exp}\left((\log(A)B)^{\ast}\right)\,\vec(W)$$
and, for $Z:=L_{\exp}\left((\log(A)B)^{\ast};W\right)B^{\ast},$ it holds
$$K_{\log}(A^{\ast})\vec(Z)=\vec\left(L_{\log}\left(A^{\ast};Z\right)\right),$$
one has
\begin{eqnarray*}
\vec\(L_f^{\star}(A,B;W)\right) &=& K^{\ast}_f(A,B)\,\vec(W),\\
& = & \left[\begin{array}{c}
		 K_{\log}(A^{\ast})\,\left((B^{\ast})^T\otimes I\right)K_{\exp}\left((\log(A)B)^{\ast}\right)\vec(W)  \\
		\left(I\otimes\log(A^{\ast})\right)K_{\exp}\left((\log(A)B)^{\ast}\right)\vec(W)
		\end{array}
		\right],
\end{eqnarray*}
and, consequentely,
\begin{equation}\label{Ladjoint}
L_f^{\star}(A,B;W)= \left[\begin{array}{c}
L_{\log}\left(A^{{\ast}};\,L_{\exp}\left((\log(A)B)^{{\ast}};W\right)B^{{\ast}}\right)\\
\log(A^{\ast})L_{\exp}\left((\log(A)B)^{\ast};W\right)
\end{array}
\right].
\end{equation}

We are now ready to propose an algorithm to estimate the condition number $\kappa_f$ of the matrix-matrix exponentiation with respect to the Frobenius norm.

{\rm \begin{alg}\label{algorithm}
Given $A,B\in\mathbb{C}^{n\times n}$, with $A$ having no eigenvalues on the closed negative real axis, this algorithm estimates the condition number $\kappa_f(A,B)$ defined in (\ref{cond-number}), where $f(A,B)=A^B$, with respect to the Frobenius norm.
\begin{enumerate}
\item[] Choose nonzero starting matrices  $E_0,F_0\in\mathbb{C}^{n\times n}$ and a tolerance \texttt{tol};
\item[] Set $\gamma_0=0$, $\gamma_1=1$ and $k=0$;
\item[] \texttt{while} $\left|\gamma_{k+1}-\gamma_k\right| > \mathtt{tol}\,\gamma_{k+1}$
\begin{itemize}
\item[] $W_{k+1}=L_f(A,B;E_k,F_k)$, with $L_f$ given by (\ref{explicit});
\item[] $Z_{k+1}=L^\star_f\left(A,B;W_{k+1}\right)$, with $L_f^\star$ given by (\ref{Ladjoint});
\item[] $\gamma_{k+1}=\left\|Z_{k+1}\right\|_F/\left\|W_{k+1}\right\|_F$;
\item[] $E_{k+1}=Z_{k+1}(1:n,1:n)$; $F_{k+1}=Z_{k+1}(n+1:2n,1:n)$;
\item[] $k=k+1$;
\end{itemize}
\item[] \texttt{end}
\item[] $\|L_f(A,B)\|_F \approx \gamma_{k+1}$;
\item[] $\kappa_f(A,B)= \|(A,B)\|_F \,\|L_f(A,B)\|_F/\left\|A^B\right\|_F$.
\end{enumerate}
\end{alg}}

\medskip
{\bf Cost}: $(5\alpha_1+\alpha_2+2\alpha_3+2\alpha_4)k$, where $\alpha_1$ is the cost of computing one matrix-matrix product (about $2n^3$), $\alpha_2$ is the cost of computing $\log(A)$, $\alpha_3$ corresponds to the computation of $L_{\log}(A;Z)$ ($Z$ stands for a given complex matrix of order $n$), and $\alpha_4$ is the cost for $L_{\exp}\left(\log(A)B;Z\right)$. If $\log(A)$ and $L_{\log}(A;Z)$ are computed by \cite[Alg. 5.1]{Mohy13}, then $(\alpha_2+2\alpha_3)k$ is about
$\left(25+\left(19+\frac{13}{3}(s+m)\right)k\right)n^3$, where $s$ is the number of square roots needed in the inverse scaling and squaring procedure and $m$ is the order of Pad\'e approximants considered; assuming that $L_{\exp}$ is evaluated by \cite[Alg. 6.4]{Mohy09}, $2\alpha_4k$ is about $(4w_m+12s+32/3)kn^3$, where $w_m$ is a number given in \cite[Table 6.2]{Mohy09}, which is related with the order of Pad\'e approximants to the matrix exponential, and $s$ is the number of squarings.

\section{Numerical experiments}\label{experiments}

We have implemented Algorithm \ref{algorithm} in MATLAB, with unit roundoff $u\approx 1.1\times 10^{-16}$, with a set of ten pairs of matrices $(A_j,B_j),\ j=1,\ldots,10$, with sizes ranging from $10\times 10$ to $15\times 15$. Many pairs include matrices with nonreal entries and/or matrices from MATLAB's gallery (for instance, the matrices \texttt{lehmer}, \texttt{dramadah}, \texttt{hilb}, \texttt{cauchy} and \texttt{condex}).

The top-left plot displays the relative errors for the condition number $\kappa_f(A_j,B_j)$ estimated by Algorithm \ref{algorithm}, for each pair of matrices. As ``exact condition number'', we have considered the value given by our implementation of Algorithm 3.17 in \cite{Higham}. It is worth noticing that this latter algorithm requires $O(n^5)$ flops while Algorithm \ref{algorithm} involves $O(n^3)$ flops. Top-right graphic shows that just $2$ iterations in Algorithm \ref{algorithm} were needed to meet the prescribed tolerance $\mathtt{tol}=10^{-1}$. The bottom-left plot illustrates our claim after the proof of Theorem \ref{thm-bound} about the small norm of the matrix logarithm and, finally, the bottom-right plots the values of the relative condition number $\kappa_f(A_j,B_j)$, for each $j$.

\begin{center}
\begin{figure}
%\centering
\hspace*{-1.9cm}\includegraphics[height=10cm]{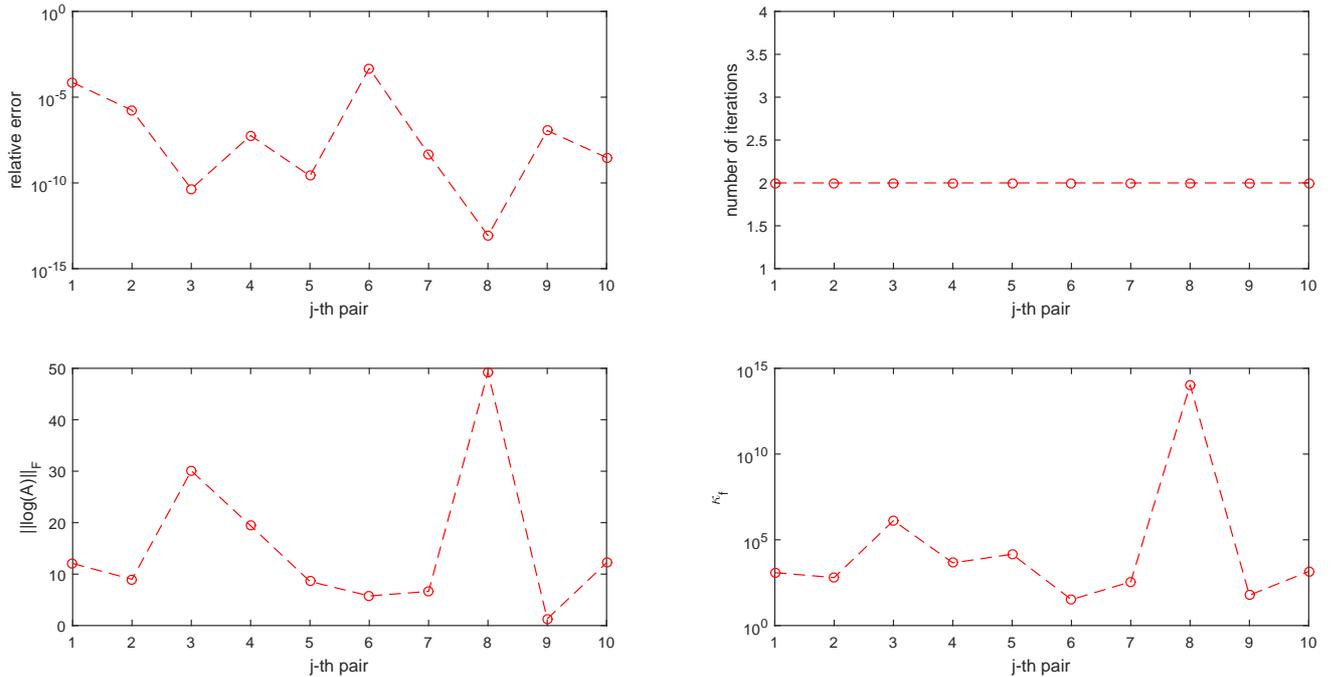}
\caption{\small Top left: Relative errors for the condition number $\kappa_f(A_j,B_j),\ j=1,\ldots,10$, estimated by Algorithm \ref{algorithm} for each pair of matrices; top right: number of iterations required by Algorithm \ref{algorithm} for a tolerance $\mathtt{tol}=10^{-1}$; bottom-left: displays $\|\log(A_j)\|_F$ for each $j$; bottom-right: relative condition number $\kappa_f(A_j,B_j)$.
 }
\label{figura}
\end{figure}
\end{center}

%\newpage
\section{Conclusions}\label{conclusions}
The Fr\'echet derivative of the matrix-matrix exponentiation and its conditioning have been investigated for the first time (as far as we know). We have given a general formula for the Fr\'echet derivative of certain bivariate matrix functions, with applications to well-know bivariate matrix functions, including the matrix-matrix exponentiation. An algorithm based on the power method for estimating the relative condition number has been proposed. Some numerical experiments illustrate our results. Basic results on the matrix-matrix exponentiation have been derived as well.
 \\\\
%\textbf{Acknowledgments}\\
%Special thanks go to the anonymous referee for some valuable suggestions, which have resulted in the improvement of this work.

%\newpage

\end{document}